\documentclass[11pt]{amsart}
\usepackage{comment}
\usepackage{amssymb,amsmath,amsthm,mathrsfs,amstext,amssymb}
\usepackage{dsfont}
\usepackage{tabularx}
\usepackage{array}
\usepackage[square,sort,comma,numbers]{natbib}
\usepackage{url} 
\usepackage{hyperref}
\usepackage{accents}
\usepackage[shortlabels]{enumitem}
\usepackage{tikz-cd}

\newtheorem{theorem}{Theorem}
\newtheorem*{theorem*}{Theorem}

\newtheorem*{maintheorem*}{Main Theorem}
\newtheorem{lemma}[theorem]{Lemma}
\newtheorem*{lemma*}{Lemma}
\newtheorem{fact}[theorem]{Fact}
\newtheorem*{fact*}{Fact}
\newtheorem{corollary}[theorem]{Corollary}
\newtheorem*{corollary*}{Corollary}
\newtheorem{proposition}[theorem]{Proposition}
\newtheorem*{proposition*}{Proposition}
\newtheorem*{gquestion*}{Guiding Question}

\theoremstyle{definition}

\newtheorem*{definition*}{Definition}
\newtheorem{remark}[theorem]{Remark}
\newtheorem*{remark*}{Remark}
\newtheorem{question}[theorem]{Question}
\newtheorem*{question*}{Question}

\newcommand{\F}{{\mathcal{F}}}

\newcommand{\I}{{\mathcal{I}}}
\newcommand{\J}{{\mathcal{J}}}

\newcommand{\M}{{\mathcal{M}}}

\newcommand{\U}{{\mathcal{U}}}

\newcommand{\ZZ}{{\mathbb{Z}}}

\renewcommand{\a}{{\mathfrak{a}}}

\DeclareMathOperator{\restr}{\upharpoonright}

\newcommand{\simpleset}[1]{{\{{#1}\}}}

\newcommand{\set}[2]{{\{ {#1} \mid {#2} \}}}

\newcommand{\infsubset}[1]{{[#1]^{\omega}}}

\newcommand{\card}[1]{|#1|}
\newcommand{\intersection}{\cap}
\newcommand{\union}{\cup}

\DeclareMathOperator{\non}{non}

\DeclareMathOperator{\bairespace}{{^\omega \omega}}
\DeclareMathOperator{\cantorspace}{{^\omega 2}}

\DeclareMathOperator{\id}{id}

\DeclareMathOperator{\fix}{fix}

\title{Coanalytic families of functions}

\author{Julia Millhouse}
\address{TU Wien, Institute for Discrete Mathematics and 
Geometry,  Wiedner Hauptstrasse 8–10, 1040 Vienna, Austria}
\email{julia.marie.millhouse@univie.ac.at}

\author{Lukas Schembecker}
\address{University of Hamburg, Department of Mathematics, Bundesstraße 55, 20146 Hamburg, Germany}
\email{lukas.schembecker@uni-hamburg.de}

\begin{document}

    \begin{abstract}
        For Van Douwen families, maximal families of eventually different permutations and maximal ideal independent families we show that the existence of a $\Sigma^1_2$ family implies the existence of a $\Pi^1_1$ family of the same size.
        We also prove a similar, but slightly weaker result for generating sets of cofinitary groups.
        
    \end{abstract}

    \maketitle

    \section{Introduction}
    Many combinatorial sets of reals defining cardinal
    characteristics can be constructed given a wellordering
    of the continuum, and hence 
    $\Sigma_2^1$ examples of such sets exist in {\sf L}, given the
    $\Sigma_2^1$-definable wellorder of the constructible reals. 
    This was initially observed by G\"odel \cite{Godel39}.
    With more careful methods, the recursive construction can 
    be done in such a way as to yield a coanalytic ($\Pi_1^1$)
    witness to the combinatorial family in question; this
    is of particular interest when the family is known 
    to not be analytic, as then the minimal complexity of such an 
    object is then decided by the coanalytic witness. 
    A robust coding technique originating in the work of 
    Erd\H{o}s, Kunen, and Mauldin \cite{Erdos1981}, later
    streamlined by Miller \cite{miller1989}, has been the 
    main tool for obtaining coanalytic witnesses of various
    combinatorial families in models of \textsf{V=L},
    see also \cite{Zoltan2014}. 
    In the literature at the intersection of 
    descriptive set theory and set theory of the reals, 
    this method has been used in conjunction with the 
    theory of preservation and indestructibility in forcing 
    in order to show that these coanalytic sets are consistent
    with $\neg$\textsf{CH} or even that some cardinal
    characteristic inequality of interest can be realized with
    a coanalytic witness for the cardinal of value 
    $\aleph_1$; see, for example, \cite{bergfalk2022projective}, \cite{FSTcofin}, \cite{BFK19}, or 
    \cite{FischerSchrittesser_2021}.
  
    More recently, T\"ornquist \cite{Tornquist_2013}
    has constructed a coanalytic mad family under weaker 
    assumptions than \textsf{V=L}; namely, he shows that 
    assuming there exists a $\Sigma_2^1$ mad family, 
    then there exists a coanalytic mad family. His proof
    is purely combinatorial and provides a simpler way to obtain
    the aforementioned applications of Miller's method 
    because of its ability to be applied in models of
    $\neg$\textsf{CH}. Statements and proofs 
    resembling T\"ornquist's \cite{Tornquist_2013}
    subsequently were given
    for other combinatorial sets of reals such as
    maximal independent families,
    maximal eventually different families, and ultrafilter bases
    (see \cite{BFK19},
    \cite{FischerSchrittesser_2021}, and \cite{schilhanUF}, respectively). In \cite{Millhouse24},
    these and further examples of where this implication appears 
    in
    the literature are presented in a general framework following the initial structure of the proof for 
    mad families.
  
    Working with this framework, in this paper we consider four
    cases of combinatorial sets of reals: Van 
    Douwen families (Section 2), maximal eventually different 
    families of permutations (Section 3), maximal ideal independent families (Section 5), and
    maximal cofinitary groups (Section 4). In the first 
    three of this list, we show that the existence of a 
    $\Sigma_2^1$ such family 
    implies the existence of a coanalytic family of the 
    same size; in the case of maximal cofinitary groups, 
    we obtain the analagous result for 
    the generating set of a maximal cofinitary group
    which is also maximal as an eventually different
    family of permutations.

    \section{Van Douwen families}

     A family $F \subseteq \bairespace$ is
    called \emph{eventually different} if 
    $\card{f \intersection g} < \omega$ for all distinct
    $f,g \in F$, i.e., there are only finitely many 
    $n$ such that $f(n) = g(n)$. Such a family $F$ is \emph{maximal} if it is maximal with respect
    to inclusion among all eventually different families; 
    equivalently, for any $f \in \omega^\omega$ there exists
    $g \in F$ such that $\card{f \intersection g} = 
    \omega$. A strengthening of this  notion of maximality
    is that of being Van Douwen: we say an eventually different
    family 
    $F$ is \emph{Van Douwen} if for any infinite partial
    function $f \in \bairespace$, there exists 
    $g \in F$ with $\card{f \intersection g} = 
    \omega$. In other words, Van Douwen families are
    maximal eventually different families which are also
    maximal with respect to infinite partial functions. 
    
    Both \textsf{CH} and \textsf{MA} imply the existence
    of Van Douwen families, and Zhang \cite[Theorem 4.2]{Zhang99} 
    shows that under \textsf{CH}, there exists a 
    Cohen-indestructible Van Douwen family.
    Later, Raghavan \cite{Raghavan_2010} proved that there always is a Van Douwen family, answering a question by Van Douwen---hence the naming. 
    Regarding definability in the sense of the 
    projective hierarchy, \cite{Raghavan_2010} also shows that Van Douwen families can never be analytic,
     in stark contrast to the situation for maximal eventually different families. Indeed, Shelah and Horowitz \cite{HorowitzShelah_2024}
	constructed under \textsf{ZFC} 
	a Borel maximal eventually 
	different family of size $\mathfrak{c}$. As Borel and analytic sets satisfy the perfect 
	set property and maximal eventually different families
	cannot be countable, any Borel or 
	analytic maximal eventually different family must 
	always be of size $\mathfrak{c}$, 
    leaving open the question as to the projective 
    definability of maximal eventually different families of size 
	strictly less than $\mathfrak{c}$ in models of
	$\neg$\textsf{CH}.
    
	In \cite{FischerSchrittesser_2021} Fischer and Schrittesser construct a 
	maximal eventually different family 
	indestructible by a countable support iteration or product of 
	Sacks-forcing of any length, and that moreover
    a coanalytic such family exists in {\sf L}. 
	Specifically, the construction of 
    \cite{FischerSchrittesser_2021} can be carried
    out in {\sf L} in a $\Sigma_2^1$ way, and 
	\cite[Theorem 8]{FischerSchrittesser_2021} establishes
	that the existence of a $\Sigma_2^1$ maximal eventually different 
    family
	is equivalent to the existence of a coanalytic 
	such family; this last
    argument follows the structure of Törnquist's original \cite{Tornquist_2013} for the case of mad families.
    In the context of functions, \cite{FischerSchrittesser_2021} 
    obtains a $\Pi_1^1$ set from one which is $\Sigma_2^1$
    by a continuous coding of two reals $(f,y) \in 
    \bairespace \times \bairespace$ into one $g \in \bairespace$,
    so that when $(f,y) \mapsto g$ is applied to the 
    $\Pi_1^1$ subset of which the $\Sigma_2^1$ med family is 
    a projection, the resulting image is still maximal 
    eventually different. 
    This is achieved by directly coding both $f$ and the real $y$ into the function 
	values of $g$. 
    However, their coding destroys the property of being 
    Van Douwen, and so the equivalence 
    between $\Sigma_2^1$ and $\Pi_1^1$ Van Douwen families
    required a different coding argument, used below. 
    
	\begin{theorem}\label{THM_vandouwen}
		If there is a $\Sigma^1_2$ Van Douwen family, then there is a $\Pi^1_1$ Van Douwen family of the same size.
	\end{theorem}

	\begin{proof}
		Define functions $\chi_0, \chi_1 : \bairespace  \times \cantorspace \to \bairespace$
		\begin{align*}
			\chi_0(f,c)(n) &:= 2 f(n) + c(n),\\
			\chi_1(f,c)(n) &:= 2 f(n) + 1 - c(n).
		\end{align*}
		Now, assume that $F$ is a $\Sigma^1_2$ Van Douwen family,
		and let $H \subseteq \bairespace  \times \cantorspace$ be $\Pi^1_1$ such that $F$ is the projection of $H$ to the first component.
		By uniformization (see, for example, 
        \cite[Theorem 4E.4]{moschovakis1980}), we may assume that $H$ is the graph of a partial function.
		Let
		\[
			G := \chi_0[H] \cup \chi_1[H].
		\]
		We claim that $G$ is the desired Van Douwen family.
		So let $f_0 \in F$ and $c_0 := H(f_0)$.
		By construction we have for all $n \in \omega$ 
		\[
			\chi_0(f_0,c_0)(n) \neq \chi_1(f_0,c_0)(n).
		\]
		Similarly, for $f_1 \in F$ with $f_1 \neq f_0$ and $c_1 := H(f_1)$ we may choose $N \in \omega$ such that $f_0(n) \neq f_1(n)$ for all $n > N$.
		But then for all such $n > N$ and $i_0, i_1 \in 2$ we also have
		\[
			\chi_{i_0}(f_0, c_0)(n) \neq \chi_{i_1}(f_1, c_1)(n).
		\]
		Thus, all members of $G$ are eventually different.
		Now, let $g:A \to \omega$ be an infinite partial function.
		Then, we define $\hat{g}:A \to \omega$ by
		\[
			\hat{g}(n) := \lfloor\frac{1}{2}g(n)\rfloor.
		\]
		Since $F$ is Van Douwen, choose $f \in F$ and $B \in \infsubset{A}$ such that $f \restr B = \hat{g} \restr B$.
		Let $c := H(f)$, then for all $n \in B$ there are $i,j \in 2$ with
		\[
			g(n) = 2 \hat{g}(n) + i = 2 f(n) + i = \chi_j(f,c)(n).
		\]
		Thus, either $g =^\infty \chi_0(f,c)$ or $g =^\infty \chi_1(f,c)$ as desired.        
		As for the definability of $G$, we will show that 
		for all $g \in \bairespace$, 
		$$ g \in G \Leftrightarrow
		\exists (f,c) \in \Delta_1^1(g) [ 
		(f,c) \in H \wedge ( \chi_0(f,c) = g \vee \chi_1(f,c) 
		= g)],$$
		which is a $\Pi_1^1$ definition by 
		the Spector-Gandy theorem (see, for example, 
		\cite[Corollary 29.3]{miller1989}). 
		Indeed, given any $g$, we have that 
		$f = \lfloor \frac{g}{2} \rfloor$, where 
		  \[
            \lfloor \frac{g}{2} \rfloor(n) = 
            \begin{cases}
                \frac{g(n)}{2} & \text{if } g(n) \text{ even},\\
                \frac{g(n) - 1}{2} & \text{if }g(n) \text{ odd}.
            \end{cases}
        \]
	    Clearly $g \mapsto \lfloor \frac{g}{2} \rfloor$ is a recursive function. 
        Then we can define the reals
        $c(n) = i$ if and only if $g(n) \mod 2 = i$,
        and $c'(n) = 1-c(n)$. 
        To check whether $(f, c) \in H$ or 
        $(f,c') \in H$ are both $\Pi_1^1$, 
        and checking 
        $\chi_i(f,c) = g$ is Borel for each $i < 2$. 
        This shows the $\Pi_1^1$-definability of $G$ above. 
	\end{proof}

\begin{corollary}
It is consistent with $\mathfrak{c} \geq \aleph_2$ that 
there exists a coanalytic Van Douwen family of size 
$\aleph_1$.
\end{corollary}

\begin{proof}
    In a model of $V=L$, repeat the construction of \cite{Zhang99} to obtain a $\Sigma_2^1$ Cohen-indestructible 
    Van Douwen family $F$, utilizing an enumeration 
    of the set of nice Cohen names for reals given by the
    $\Sigma_2^1$-definable wellorder $\leq_L$. 
    Let $\mathbb{P}$ be the forcing for adding $\kappa$-many 
    Cohen reals,
    where $\kappa \geq \aleph_2$ is a regular cardinal,
    and let $G$ be $\mathbb{P}$-generic over $V$. 
    By indestructibility, $F$ is still a Van Douwen family 
    in $V[G]$, and by Shoenfield absoluteness $F$ remains
    $\Sigma_2^1$-definable. Therefore Theorem 
    \ref{THM_vandouwen} applies.
\end{proof}

    \section{Eventually different families of permutations}

      In this section we will consider eventually different
     families of permutations, which are eventually different
    families $F \subseteq S_\infty$, where 
    $S_\infty$ denotes the set of permutations (i.e. bijections)
     of $\omega$, and moreover which are maximal as an eventually
    different family of permutations---that is,
    for any 
    $g \in S_\infty$, there is $f \in F$ and 
    infinitely many $n \in \omega$ for which $f(n) = g(n)$.
    The minimal size of a maximal eventually different
    family of permutations is denoted $\mathfrak{a}_p$. 
    It is consistent that $\mathfrak{a}_e = 
    \mathfrak{a}_p = \aleph_1 < \mathfrak{c} = \aleph_2$ and 
    moreover there exist coanalytic witnesses for $\mathfrak{a}_e,
    \mathfrak{a}_p = \aleph_1$;
    see \cite{FischerSwitzer_2023}. We note that it remains 
    open whether $\mathfrak{a}_e = \mathfrak{a}_p$ is a 
    theorem of \textsf{ZFC}. 

    The results of Horowitz and Shelah give maximal 
    eventually different families of permutations can be
    Borel/analytic, so as explained above, we are
    interested in the projective definability of 
    witnesses for $\mathfrak{a}_p$ in models 
    of $\mathfrak{a}_p < \aleph_2 \leq \mathfrak{c}$,
    as these latter families cannot be analytic. 
    Below we show that a coanalytic 
    witness for $\mathfrak{a}_p = \aleph_1$ is equivalent 
    to the existence of 
    a $\Sigma^1_2$ witness for $\a_p = \aleph_1$, and 
    hence $\Pi_1^1$ is the best definability possible in this 
    context.
    The proof will make use of the following easy to prove graph theoretic fact:

    \begin{fact}
        Every bipartite $2$-regular graph decomposes into a disjoint union of cycles of even or infinite length and hence has a perfect matching by picking edges alternatingly.
    \end{fact}

	\begin{lemma}\label{LEM_2to1Bijection}
		Assume $f:\omega \to \omega$ is $2$-to-$1$, i.e.\ every $n \in \omega$ has exactly two preimages.
		Then there is a function $i:\omega \to 2$ such that the function $g:\omega \to \omega$ defined by $g(n) := f(2n + i(n))$ is a bijection.
	\end{lemma}

	\begin{proof}

        Consider the following bipartite graph $H$ (with possible multi-edges):
        \begin{enumerate}[(i)]
            \item We have countably many left $\set{L_n}{n \in \omega}$ and right $\set{R_n}{n \in \omega}$ nodes,
            \item For each $n \in \omega$ we have an edge $e_n$ between $L_{\lfloor\frac{n}{2}\rfloor}$ and $R_{f(n)}$.
        \end{enumerate}
        
        By construction, every $L_n$ has degree $2$.
        As $f$ is $2$-to-$1$ the same holds for the $R_n$, i.e.\ $H$ is $2$-regular.
        Hence, $H$ has a perfect matching $P$ by the fact above.
        Now, we define for $n \in \omega$
        \[
            i(n) := 
            \begin{cases}
                0 & \text{if } e_{2n} \in P,\\
                1 & \text{if } e_{2n + 1} \in P.
            \end{cases}
        \]
        Note, for every $n \in \omega$ that $e_{2n}$ and $e_{2n+1}$ are the only edges incident to $L_n$.
        Thus, exactly one of these cases above occurs as we have a perfect matching.
        It is also easy to see that $g$ will then be bijective: If $g$ was not injective, then $P$ would not be a matching, and if $g$ was not surjective, then $P$ would not be perfect.
    \end{proof}

	\begin{theorem} \label{THM_CoanalyticMEDP}
		If there is a $\Sigma^1_2$ maximal eventually different family of permutations, then there is a $\Pi^1_1$ maximal eventually different family of permutations of the same size.
	\end{theorem}

	\begin{proof}
		Define functions $\chi_0, \chi_1 :S_\infty  \times \cantorspace \to S_\infty$ for $n \in \omega$ by
		\begin{align*}
			\chi_0(f,c)(2n) &:= 2 f(n) + c(n),\\
            \chi_0(f,c)(2n + 1) &:= 2 f(n) + 1 - c(n),\\[5pt]
            \chi_1(f,c)(2n) &:= 2 f(n) + 1 - c(n),\\
            \chi_1(f,c)(2n + 1) &:= 2 f(n) + c(n).
		\end{align*}
		Now, assume that $F$ is a $\Sigma^1_2$ maximal family of permutations.
		Further let $H \subseteq \bairespace  \times \cantorspace$ be $\Pi^1_1$ such that $F$ is the projection of $H$ to the first component.
		By uniformization we may assume that $H$ is the graph of a partial function.
		Let
		\[
			\hat{F} := \chi_0[H] \cup \chi_1[H].
		\]
		We claim that $\hat{F}$ is the desired maximal eventually different family of permutations.
		Let $f \in S_\infty$, $c \in \cantorspace$ and $i \in 2$.
		First we show that $\chi_i(f, c) \in S_\infty$.
		Assume for $n_0, n_1 \in \omega$ and $j_0, j_1 \in 2$ we have
		\[
			\chi_i(f, c)(2n_0 + j_0) = \chi_i(f,c)(2n_1 + j_1)
		\]
		By construction, this implies that $f(n_0) = f(n_1)$.
		But $f$ is injective, so also $n_0 = n_1$.
		Again, by construction of $\chi_i(f,c)$ we also obtain $j_0 = j_1$, i.e.\ $\chi_i(f,c)$ is injective.
		For surjectivity, let $m \in \omega$ and $j \in 2$.
		Since $f$ is surjective, choose $n$ such that $f(n) = m$.
		But then either
		\[
			\chi_i(f,c)(2n) = 2 f(n) + j = 2m + j
		\]
		or we have
		\[
			\chi_i(f,c)(2n + 1) = 2 f(n) + j = 2m + j.
		\]
		Hence $\chi_i(f,c)$ is surjective.
		Next, let $f_0 \in F$ and $c_0 := H(f_0)$.
		By construction we have for all $n \in \omega$
		\[
			\chi_0(f_0, c_0)(n) \neq \chi_1(f_0, c_0)(n).
		\]
		Similarly, for $f_1 \in F$ with $f_1 \neq f_0$ and $c_1 := H(f_1)$ we may choose $N \in \omega$ such that $f_0(n) \neq f_1(n)$ for all $n > N$.
		But then for all such $n > N$ and $i_0, i_1 \in 2$ we also have
		\[
			\chi_{i_0}(f_0, c_0)(n) \neq \chi_{i_1}(f_1, c_1)(n).
		\]
		Thus, all members of $\hat{F}$ are eventually different.
		Now, towards maximality of $\hat{F}$ let $\hat{g} \in S_\infty$.
		Choose a function $i:\omega \to 2$ such that $g:\omega \to \omega$ defined by
		\[
			g(n) := \lfloor\frac{\hat{g}(2n + i(n))}{2}\rfloor
		\]
		is a bijection.
		This is possible by the previous lemma as the function $\lfloor\frac{\hat{g}(n)}{2}\rfloor$ is $2$-to-$1$.
		By maximality of $F$, we may choose $f \in F$ and $A \in \infsubset{\omega}$ such that $g \restr A = f\restr A$.
		Let $c := H(f)$ and $n \in A$.
		Then there are $j, k\in 2$ so that
		\[
			\hat{g}(2n + i(n)) = 2g(n) + j = 2f(n) + j = \chi_k(f,c)(2n + i(n)).
		\]
		Thus, either $\hat{g} =^\infty \chi_0(f,c)$ or $\hat{g} =^\infty \chi_1(f,c)$ as desired.
        As before Spector-Gandy shows that $\hat{F}$ is $\Pi^1_1$ as for fixed $i \in 2$ we can compute $(f,c)$ from $\chi_i(f,c)$.
	\end{proof}

    Before we move on to maximal cofinitary groups, we discuss the minimal complexity of a maximal family of permutations.
    For $\a_e$, Schrittesser \cite{Schrittesser_2017} showed that there is a $\Pi^0_1$, i.e.\ a closed maximal eventually different family.
    Similarly, for $\a_g$ (see the following section), Mejak and Schrittesser showed that there is a $\Pi^0_1$ set freely generating a maximal cofinitary group.
    Thus, the whole group has complexity $\Sigma^0_2$.
    Moreover, this group is not only maximal as a cofinitary group, but also maximal as an eventually different family of permutations; see \cite[Proposition 2.13]{MejakSchrittesser_2022}.
    Hence, they also proved that there is a $\Sigma^0_2$ witness for $\a_p$, however to the knowledge of the authors it is not known what the minimal complexity for a maximal eventually different family of permutations is.

    \begin{question}
        Is there a $\Pi^0_1$ maximal eventually different family of permutations?
    \end{question}

    In the same way Schrittesser obtained a $\Pi^0_1$ maximal eventually different family, one might assume that our proof above may be used to obtain an analogous statement to \cite[Lemma 4.1]{Schrittesser_2017}.
    However, with our coding above we only get the following:

    \begin{lemma}
        Let $0 < \xi < \omega_1$.
        If there is a $\Pi^0_{\xi + 2}$ maximal eventually different family of permutations, then there is a $\Pi^0_{\xi + 1}$ maximal eventually different family of permutations.
    \end{lemma}

    \begin{proof}
        Adapt \cite[Lemma 4.1]{Schrittesser_2017} using the coding above.
    \end{proof}

    Note the extra assumption of $0 < \xi$, so with this lemma we can only obtain a $\Pi^0_2$ witness for $\a_p$.
    Essentially, this is due to the fact we are coding elements
    of $2^{\omega}$ instead of $\omega^\omega$, as done in 
    \cite{Schrittesser_2017}.
    Hence, in order to define the maximal eventually different family of permutations of lower complexity with the approach
    above, one needs to express that the sequence of (non-)flips encodes a sequence of natural numbers.
    However, to this end we need to require that the sequence of (non-)flips is not eventually constant.
    But this is a $\Pi^0_2$ statement, thus requiring the extra assumption.

  \section{Cofinitary groups}

    Next, we will consider maximal cofinitary groups.
    If an eventually different family of permutations $G$ is also a group with respect to concatenation, then we call it a \textit{cofinitary group}.
    Equivalently, every element of $G$ is either the identity or only has finitely many fix-points.
    $G$ is \textit{maximal} if its maximal with respect to inclusion among all cofinitary groups.
    The minimal cardinality of a maximal cofinitary group is uncountable and denoted with $\a_g$; again no known relations or the absence thereof between $\a_e, \a_p$ and $\a_g$ are known.
    
    In terms of definability it is often easier to obtain a definable generating set for a cofinitary group.
    We say that $F$ \textit{generates} $G$ if $\langle F \rangle = G$, where $\langle F \rangle$ is the group generated by $F$.
    If there are no relations among the generators in $F$, we say that $F$ \textit{freely generates} $G$.
    For example, it was first shown by Gao and Zhang \cite{GaoZhang_2008} that in {\sf L} there is a $\Pi^1_1$ generating set for a maximal cofinitary group, and shortly
    thereafter Kastermans \cite{Kastermans_2009} showed that indeed the entire group can be $\Pi^1_1$.
    Horowitz and Shelah \cite{HorowitzShelah_2025} then showed that there always is a Borel maximal cofinitary group.
    More recently, Fischer, Schrittesser and the second author \cite{FischerSchembeckerSchrittesser_2023} showed that in {\sf L} there is a $\Pi^1_1$ cofinitary group which is indestructible by various different tree forcings preserving $\non(\M)$.
    They employed an intricate coding argument, where information is coded into the lengths orbits and into the amount of orbits of a certain length.

    Here, we use a simpler coding technique to obtain a result for cofinitary groups similar to the previous sections.
    However, with our methods we can only get a result for generating sets of cofinitary groups and we need to additionally assume that the maximal cofinitary group is freely generated and also maximal as an eventually different family of permutations.
    This seems like a strong extra assumption, but indeed most cofinitary groups constructed in the natural way satisfy these assumptions.
    In particular, the next theorem provides a different way to show that many {\sf L}-extensions have a coanalytic generating set for a maximal cofinitary group. 

    \begin{theorem}\label{THM_CoanalyticMCG}
        If there is a $\Sigma^1_2$ family freely generating a maximal cofinitary group, which is also maximal as an eventually different family of permutations, then there is a $\Pi^1_1$ family generating a maximal cofinitary group, which is also maximal as an eventually different family of permutations.
    \end{theorem}

    \begin{proof}\renewcommand{\qedsymbol}{}
        As in Theorem~\ref{THM_CoanalyticMEDP} define $\chi_0, \chi_1 :S_\infty  \times \cantorspace \to S_\infty$ for $n \in \omega$ by
		\begin{align*}
			\chi_0(f,c)(2n) &:= 2 f(n) + c(n),\\
            \chi_0(f,c)(2n + 1) &:= 2 f(n) + 1 - c(n),\\[5pt]
            \chi_1(f,c)(2n) &:= 2 f(n) + 1 - c(n),\\
            \chi_1(f,c)(2n + 1) &:= 2 f(n) + c(n).
		\end{align*}
        This time, assume that $F$ is a $\Sigma^1_2$ set freely generating the maximal cofinitary group $\Gamma := \langle{ F \rangle}$.
		Further let $H \subseteq \bairespace  \times \cantorspace$ be $\Pi^1_1$ such that $F$ is the projection of $H$ to the first component.
		By uniformization we may assume that $H$ is the graph of a partial function.
		Let
		\[
			\hat{F} := \chi_0[H] \cup \chi_1[H].
		\]
        By the arguments in the previous section, $\hat{F}$ is a family of permutations and is $\Pi^1_1$.
        It remains to show that for $\hat{\Gamma} := \langle \hat{F}\rangle$ we have
        \begin{enumerate}
            \item $\hat{\Gamma}$ is cofinitary,
            \item $\hat{\Gamma}$ is maximal as an eventually different family of permutations.
        \end{enumerate}
        From now on we consider the partition of $\omega$ given by the pairs $B_n := \simpleset{2n, 2n+1}$ for $n \in \omega$.
        We need the following lemmata:
    \end{proof}

    \begin{lemma}\label{LEM_Homomorphism}
        For every $\hat{g} \in \hat{\Gamma}$ there is a unique $g \in \Gamma$ such that for all $n \in \omega$
        \[
            \hat{g}[B_n] = B_{g(n)}.
        \]
        Moreover, the assignment $\Psi:\hat{\Gamma} \to \Gamma$ given by $\hat{g} \mapsto g$ is a surjective group homomorphism with
        $\Psi(\chi_i(f, H(f))) = f$ for all $f \in F$ and $i \in 2$.
    \end{lemma}

    \begin{proof}
        For uniqueness, suppose $g,h \in \Gamma$ satisfy $\hat{g}[B_n] = B_{g(n)} = B_{h(n)}$.
        Since the $B_n$'s are disjoint this implies $g = h$, so it suffices to prove existence.
        To this end, for $i \in 2$, $f \in F$ and $c := H(f)$, by definition of $\chi_i$ we clearly have that
        \[
            \chi_i(f, c)[B_n] = B_{f(n)}.
        \]
        Similarly, it is easy to see that also for the inverse $\chi_i(f, c)^{-1}$ we have
        \[
            \chi_i(f, c)^{-1}[B_n] = B_{f^{-1}(n)}.
        \]
        Thus, the required $g \in \Gamma$ exists for all generators and inverses thereof in $\hat{\Gamma}$.
        We prove the general case by induction on the length of $\hat{g}$, which we can consider as a reduced
        word in the alphabet $\hat{F}^{\pm1}$, so 
        let $\hat{g} = \hat{f}_1 \dots \hat{f}_k$ be such that 
        $\hat{f}_i \in \hat{F}^{\pm1}$ for each $i = 1, \dots, k$.
        If $k = 0$ then $\hat{g} = \id_\omega$ and $\id_\omega[B_n] = B_n = B_{\id_\omega(n)}$ with $\id_\omega \in \Gamma$.
        Now, let $k > 0$, $\hat{h}_1 = \hat{f}_1$ and $\hat{h}_2 = \hat{f}_2 \dots \hat{f}_k$.
        By induction, there are $h_1, h_2 \in \Gamma$ such that $\hat{h}_1[B_n] = B_{h_1(n)}$ and $\hat{h}_2[B_n] = B_{h_2(n)}$ for all $n \in \omega$.
        Then, for $n \in \omega$ we compute
        \[
            \hat{g}[B_n] = \hat{h}_1[\hat{h}_2[B_n]] = \hat{h}_1[B_{h_2(n)}] = B_{h_1(h_2(n))} = B_{(h_1 \circ h_2)(n)},
        \]
        proving the existence of the desired $g := h_1 \circ h_2$.
        This computation also proves the homomorphism property of $\Psi$.
        Furthermore, we obtain
        \[
            \Psi[\hat{\Gamma}] = \Psi[\langle \hat{F} \rangle] = \langle \Psi[\hat{F}] \rangle = \langle F \rangle = \Gamma,
        \]
        so $\Psi$ is surjective.
    \end{proof}

    \begin{lemma}\label{LEM_FlipAndKernel}
        Let $\tau:\omega \to \omega$ be the flip map defined for $n \in \omega$ by
        \[
            \tau(2n) := 2n +1 \quad \text{ and } \quad \tau(2n + 1) := 2n.
        \]
        Then $\tau \in \hat{\Gamma}$ is central in $\hat{\Gamma}$ and we have $\ker(\Psi) = \simpleset{\id_\omega, \tau}$.
    \end{lemma}

    \begin{proof}
        First, note that $\tau^2 = \id_\omega$ and for every $f \in F$, $i \in 2$ and $c := H(f)$ we have
        \[
            \chi_0(f,c) \chi_1(f,c)^{-1} = \tau \quad \text{ and } \quad \chi_1(f,c)^{-1} \chi_0(f,c) = \tau.
        \]
        This shows that $\tau \in \hat{\Gamma}$ and that $\tau$ commutes with all generators of $\hat{\Gamma}$.
        Thus, $\tau$ is central in $\hat{\Gamma}$.
        Furthermore, $\Psi(\tau) = \id_\omega$, so $\ker(\Psi) \supseteq \simpleset{\id_\omega, \tau}$.
        Now, let $\hat{g} \in \ker(\Psi)$.
        Write $\hat{g}$ as a word in letters from $\hat{F}^{\pm1}$.
        Using $\chi_1^{\pm1} = \tau\chi_0^{\pm1} = \chi_0^{\pm1} \tau$, centrality of $\tau$ and $\tau^2 = \id_\omega$, replace each occurrence of $\chi_1$ by $\chi_0$ and move the $\tau$ to the front to obtain
        \[
            \hat{g} = \tau^i \hat{h},
        \]
        where $i \in 2$ and $\hat{h}$ is a word in the alphabet $\set{\chi_0(f, H(f))^{\pm1}}{ f \in F}$.
        Remember $\hat{g}, \tau \in \ker(\Psi)$, so apply $\Psi$ to obtain
        \[
            \id_\omega = \Psi(\hat{g}) = \Psi(\tau^i\hat{h}) = \Psi(\tau^i)\Psi(\hat{h}) = \Psi(\hat{h}).
        \]
        But by Lemma~\ref{LEM_Homomorphism} $\Psi$ is a homomorphism which maps $\chi_0(f, H(f))$ to $f$, so $\Psi(\hat{h})$ is the corresponding word in the generators $F^{\pm1}$ of $\Gamma$.
        But $\Gamma = \langle F \rangle$ is freely generated, so the equation above implies that $\hat{h} = \id_\omega$
        Consequently, $\hat{g} = \tau^i \in \simpleset{\id_\omega, \tau}$.
    \end{proof}

    \begin{proposition}\label{PROP_CofinitaryGroup}
        $\hat{\Gamma}$ is cofinitary.
    \end{proposition}

    \begin{proof}
        Let $\hat{g} \in \hat{\Gamma} \setminus \simpleset{\id_\omega}$.
        If $\hat{g} \in \ker(\Psi)$, then by the previous lemma we can only have $\hat{g} = \tau$ which has no fixpoints.
        Thus, we may assume that $g := \Psi(g) \in \Gamma \setminus \simpleset{\id_\omega}$.
        Let $k \in \omega$ be a fixed point of $\hat{g}$, say $k \in B_n$, then $\hat{g}[B_n] = B_n$.
        So by Lemma~\ref{LEM_Homomorphism} we have $B_{g(n)} = \hat{g}[B_n] = B_n$.
        Thus, $n = g(n)$ and we get
        \[
            \fix(\hat{g}) \subseteq \bigcup_{n \in \fix(g)} B_n.
        \]
        But $g$ is cofinitary and each $B_n$ has size $2$, so $\left|\fix(\hat{g})\right| \leq 2 \left|\fix(g)\right| < \infty$.
    \end{proof}

    \begin{proposition}
        $\hat{\Gamma}$ is maximal as an eventually different family of permutations.
    \end{proposition}

    \begin{proof}
        Let $\hat{h} \in S_\infty$ and by Lemma~\ref{LEM_2to1Bijection} choose $i:\omega \to 2$, so that
        \[
            h(n) := \lfloor \frac{\hat{h}(2n + i(n))}{2} \rfloor
        \]
        is a bijection.
        By maximality of $\Gamma$ as an eventually different family of permutations there is $g \in \Gamma$ and $A \in \infsubset{\omega}$ such that $h \restr A = g \restr A$.
        Let $\hat{g}$ be the corresponding word in $\hat{\Gamma}$, where every occurrence of $f^{\pm1} \in F^{\pm1}$ appearing
        in $g$ is replaced by $\chi_0(f, H(f))^{\pm1}$.
        Thus, we have $\Psi(\hat{g}) = g$, so as before, for every $n \in A$ there are $j,k \in 2$ such that
        \[
            \hat{h}(2n + i(n)) = 2h(n) + j = 2g(n) + j = \tau^k\hat{g}(2n + i(n)).
        \]
        But this implies either $\hat{g} =^\infty \hat{h}$ or $\hat{g} =^\infty \tau \hat{h}$, as desired.
    \end{proof}

    \begin{remark}
        The generated group $\hat{\Gamma}$ is not free, but by the considerations above, its isomorphism type is given by the product $\mathbb{\ZZ} / 2 \times F$, where $(1, \id_\omega)$ corresponds to the element $\tau$.
    \end{remark}

    Note, that the proof methods in this section crucially required $\Gamma$ to be freely-generated.
    Dropping this extra assumption likely needs a completely different proof idea.

	\begin{question}
		Does an analogous result to Theorem \ref{THM_CoanalyticMCG} hold for non-freely generated groups or even better, can such a result be proven for entire groups and not just their generating sets?
	\end{question}

      \section{Ideal independent families}
    A family $\I \subseteq \infsubset{\omega}$ is 
    \emph{ideal independent} if for all $F \in [\I]^{< \omega}$
    and $a \in \I \setminus F$, it is not the case that 
    $a \subseteq^* \bigcup F$. 
    An ideal independent family $\I$ is \emph{maximal 
    ideal independent} if $\I$ is maximal with respect to 
    inclusion. Equivalently, for every $b \in \infsubset{\omega}$
    there exists a finite $F \subseteq \I$ such that one 
    of the following occurs: 
    \begin{itemize}
        \item $b \subseteq^* \bigcup F$, or
        \item there is $a \in \I \setminus F$ such that 
        $a \subseteq^* b \union \bigcup F$. 
    \end{itemize}
    The cardinal $\mathfrak{s}_{mm}$ is defined as the minimum
    size of a maximal ideal independent family; results 
    regarding
    the relations between $\mathfrak{s}_{mm}$ and 
    other cardinal invariants can be found in \cite{CGM21} and 
    \cite{BCFS25}. Clearly maximal ideal independent families 
    can be obtained using the Axiom of Choice, however as of now 
    the minimal projective complexity of such a family is 
    unknown. The theorem below gives an upper bound of $\Pi_1^1$.

    \begin{theorem} \label{THM_coanalyticsmm}
        If there exists a $\Sigma^1_2$ maximal ideal-independent family,
        then there exists a $\Pi^1_1$ maximal ideal-independent family of the same size.
    \end{theorem}
    
    \begin{proof}
        Let $\I$ be a $\Sigma_2^1$ maximal ideal independent
        family, and let $H \subseteq \infsubset{\omega} \times 
        \infsubset{\omega}$ be a
        coanalytic set such that 
        $x \in \I$ if and only if there exists 
        $c \in \infsubset{\omega}$ with $(x,c) \in H$; by 
        uniformization, we can assume $H$ is the graph 
        of a function. 
        Define a function $\chi: \infsubset{\omega} \times \infsubset{\omega} \to \infsubset{\omega}$ and $z \in \infsubset{\omega}$ by
        \begin{align*}
            \chi(x, c) &:= 3x \cup 3c + 1 &&= \set{3n}{n \in x} \cup \set{3n+1}{n \in c},\\
            z &:= \phantom{3xaa} 3\omega + 1 \cup 3 \omega + 2 &&= \set{3n + 1}{n \in \omega} \cup \set{3n+2}{n \in \omega}.
        \end{align*}
        Let 
        \[
            \J := \chi[H] \cup \{z\}.
        \]
        We will show $\J$ is maximal ideal independent. 
        To see it is ideal independent, let $F \in [\J]^{<\omega}$
        be arbitrary and let $x \in \J \setminus  F$.
        First, suppose $x \neq z$ and $z \not\in F$.  Let 
        $(a,c) \in (\infsubset{\omega})^2$ be such that 
        $\chi(a,c)= x$.  Then $x \setminus \bigcup F$ must be
        infinite, as it contains the set 
        $a \setminus \bigcup \mathrm{proj}_0[\chi^{-1}[F]]$,
        where $\mathrm{proj}_0 : \infsubset{\omega} \times 
        \infsubset{\omega} \to \infsubset{\omega}$ is the
        projection onto the first coordinate. 
        Now suppose $x = z$. Then $z \setminus \bigcup F$ is
        infinite, as it contains the set 
        $\set{3n+2}{n \in \omega} \subseteq z$. 
        Lastly, suppose $x \neq z$ and $z \in F$;
        by the first case, it suffices to consider the case 
        $F = \simpleset{z}$. 
        But clearly, the set  $\set{3n}{n \in a} \subseteq x$ is an 
        infinite subset disjoint from $z$.

        Towards maximality, let $b \in \infsubset{\omega}$, and 
        consider $b_0 := \set{n \in \omega}{3n \in b}$. 
        By maximality of $\I$, there
        exists $F \in [\I]^{<\omega}$ such that one of the
        following occurs: 
            \begin{enumerate}
            \item $b_0 \subseteq^* \bigcup F$, or 
            \item there exists $x \in \I \setminus F$ such 
            that $x \subseteq^* b_0 \union \bigcup F$. 
            \end{enumerate}
        Suppose we are in the first case. Then 
        $b \subseteq^* \simpleset{z} \union \bigcup \chi[(F \times \infsubset{\omega}) \cap H]$, 
        since on the one hand $b \cap 3\omega$ is almost covered by 
        $\chi[(F \times \infsubset{\omega}) \cap H]$, and on the other hand,
        \[
            b \setminus 3\omega = \simpleset{3n+1 \in b} \union 
            \simpleset{3n+2 \in b} \subseteq z.
        \]
        Otherwise, fix $x \in \I$ as in case (2), and let 
        $c$ be such that $(x,c) \in H$. 
        Then $\chi(x,c)$ is an element of $\J$ such 
        that $\chi(x,c) \subseteq^* b \union \bigcup (\chi[(F \times \infsubset{\omega}) \cap H]
        \union \simpleset{z})$, since 
        $\set{3n+1}{n \in c} \subseteq z$.
        
        We claim that $\J$ is coanalytic. 
        First, note that the set $\simpleset{z}$ is clearly 
        $\Delta_1^1$-definable. Next, we have that the set 
        $\chi[H]$ is $\Pi_1^1$, as it is defined by
        \[
            x \in \chi[H] \Leftrightarrow 
            \exists a,c \in \Delta_1^1(x) [(a,c) \in H \wedge 
            \chi(a,c) = x].
        \]
        Indeed, given $x \in \infsubset{\omega}$, 
        $a$ is reconstructible by $x$ as the set $\set{n \in \omega}{3n \in x}$.
        Similarly, also $c$ is reconstructible from $x$. 
        Then, being the union of two coanalytic sets, 
        $\J$ is coanalytic.
    \end{proof}
    
    An ultrafilter $\U$ is called a \emph{p-point} if 
    for any countable $\F \subseteq \U$, there exists 
    $Y \in \U$ such that $Y \subseteq^* X$ for every 
    $X \in \F$. 
    Recently, Bardyla, Cancino-Manriquez, Fischer, and Switzer
    definethe notion of $\U$\emph{-encompassing} ideal
    independent family, where $\U$ is an ultrafilter on 
    $\omega$ (see \cite[Definition 5.1]{BCFS25}). When such $\U$ is a p-point and under some additional assumptions, this strengthening of the maximality condition 
    for ideal independent families isolates a subclass of 
    maximal ideal independent families which are indestructible 
    by any proper, $\omega^\omega$-bounding, p-point preserving 
    forcing notion. 
    Moreover, they show that under \textsf{CH}, 
    for any p-point $\U$ there exists
    a $\U$-encompassing ideal independent family with the
    required additional assumptions for indestructibility
    \cite[Theorem 5.2]{BCFS25}. Thus, adapting their 
    construction in \textsf{L} yields a $\Sigma_2^1$-definable
    maximal ideal independent family which remains maximal under 
    a broad class of proper forcing notions. 
    One particular instance of a proper, $\omega^\omega$-bounding
    forcing notion which preserves p-points is the construction 
    of a $\Delta_3^1$ wellorder of the reals in a model of 
    $\mathfrak{c} = \aleph_2$ given by Fischer and Friedman 
    in 2010 \cite{FischerFriedman_2010}; see also \cite{bergfalk2022projective}. Using these 
    observations together with Theorem \ref{THM_coanalyticsmm},
    we 
    obtain the following:

    \begin{corollary}
        It is consistent with the existence of a 
        $\Delta_3^1$ total wellorder of the reals that
        $\aleph_1 =\mathfrak{s}_{mm} < \mathfrak{c}= \aleph_2$,
        and there is a coanalytic maximal ideal independent 
        family of size $\aleph_1$.
    \end{corollary}

	\bibliographystyle{plain}
	\bibliography{refs}
	
\end{document}